\documentclass[11pt]{amsart}

\usepackage{bm}
\usepackage{graphicx}
\usepackage{mathpazo}
\usepackage{latexsym}   
\usepackage{amssymb}    
\usepackage{amsmath}    
\usepackage{amsthm}     
\usepackage{hyperref}
\usepackage[margin=1in]{geometry}

\usepackage{diagrams}
\newarrow{to}---->
\newarrow{dashedto}{}{dash}{}{dash}>
\newarrow{mto}|--->
\newarrow{impto}===={=>}
\newarrowtail{<=}\Rightarrow\Leftarrow\Uparrow\Downarrow
\newarrow{bboth}{<=}==={=>}
\newarrow{inject}{hooka}---{vee}
\newarrow{surject}----{>>}

\DeclareMathOperator{\gu}{GU}
\DeclareMathOperator{\G}{G}
\DeclareMathOperator{\U}{U}

\def\dieu{{\mathbb D}}

\renewcommand{\bar}[1]{{\overline{#1}}}

\newtheorem{theorem}{Theorem}[section]
\newtheorem{lemma}[theorem]{Lemma}

\theoremstyle{definition}
\newtheorem{remark}[theorem]{Remark}

\newcommand{\stk}[1]{{\mathcal #1}}
\newcommand{\tilstk}[1]{{\til{\stk #1}}}
\newcommand{\barstk}[1]{{\bar{\stk #1}}}

\def\setcomp{\smallsetminus}

\def\level{{\mathbb K}}

\usepackage{tikz}
\def\calo{{\mathcal O}}
\def\ff{{\mathbb F}}
\def\ra{\rightarrow}
\def\tensor{\otimes}
\newcommand{\floor}[1]{{\lfloor #1 \rfloor}}
\newcommand{\ang}[1]{{{\langle #1 \rangle}}}
\newcommand{\half}[1]{\frac{#1}{2}}
\newenvironment{alphabetize}{\begin{enumerate}

}{\end{enumerate}}
\def\calh{{\mathcal H}}
\newcommand{\myht}[1]{{\widehat{#1}}}
\newcommand{\invlim}[1]{\lim_{\stackrel{\leftarrow}{#1}}}
\def\aff{{\mathbb A}}
\def\del{\partial}
\global\let\hom\undefined
\DeclareMathOperator{\hom}{Hom}
\newcommand{\powser}[1]{[\![#1]\!]}
\newcommand{\ubar}[1]{{\underline{#1}}}
\def\iso{\cong}
\def\calc{{\mathcal C}}
\def\rat{{\mathbb Q}}
\DeclareMathOperator{\End}{End}
\def\dual{^\vee}
\def\cross{\times}
\newcommand{\oneover}[1]{\frac{1}{#1}}
\DeclareMathOperator{\id}{id}
\def\nat{{\mathbb N}}
\newcommand{\st}[1]{\{#1\}}
\def\inject{\hookrightarrow}
\def\inv{^{-1}}
\newcommand{\til}[1]{{\widetilde{#1}}}
\DeclareMathOperator{\spf}{Spf}
\DeclareMathOperator{\Frac}{Frac}

\begin{document}
\title{Irreducibility of Newton strata in ${\rm GU}(1,n-1)$ Shimura varieties}

\author{Jeffrey D. Achter}
\email{j.achter@colostate.edu}
\address{Department of Mathematics, Colorado State University, Fort
Collins, CO 80523} 
\urladdr{http://www.math.colostate.edu/~achter}

\thanks{This work was partially supported by a grant from the Simons
  Foundation (204164).}

\maketitle

\begin{abstract}
Let $L$ be a quadratic imaginary field, inert at the rational prime $p$.  Fix an integer $n\ge 3$, and let $\stk M$ be the moduli space (in characteristic $p$) of principally polarized abelian varieties of dimension $n$ equipped with an action by $\calo_L$ of signature of $(1,n-1)$.  We show that each Newton stratum of $\stk M$, other than the supersingular stratum, is irreducible.
\end{abstract}

\section{Introduction}

For a complex abelian variety $X$, the isomorphism class of its
$p$-torsion group scheme $X[p]$ and of its $p$-divisible group
$X[p^\infty]$ depend only on the dimension of $X$.  In contrast, in
characteristic $p$, there are different possibilities for the
corresponding isomorphism (or even isogeny) class.  Each such
invariant provides a stratification of a family of abelian varieties
in positive characteristic.

The isogeny class of $X[p^\infty]$ is called the Newton polygon of
$X$.  The goal of the present note is to prove that the space of
abelian varieties with given Newton polygon and a certain, specified
endomorphism structure is irreducible.

More precisely, let $L$ be a quadratic imaginary field, inert at the rational prime
$p$.  Fix an integer $n \ge 3$, and let $\stk M$ be the moduli space
(over $\ff_{p^2}$) of principally polarized abelian varieties of
dimension $n$ equipped with an action by $\calo_L$ of signature
$(1,n-1)$.  Our main result is:

\begin{theorem}
\label{thmain}
Let $\xi\not = \sigma$ be an admissible Newton polygon for $\stk M$
which is not supersingular.  Then the corresponding stratum $\stk
N^\xi$ is irreducible.
\end{theorem}

The proof of Theorem \ref{thmain} is modelled on, but considerably easier
than, that of \cite[Thm.\ A]{chaioort11}.  This is possible because
the Newton and Ekedahl-Oort stratifications on $\stk M$ are much
simpler than those of $\stk A_g$.

In the special case where $L = \rat(\zeta_3)$ and $n$ is $3$ or $4$,
$\stk M$ essentially coincides with a component of the moduli space of
cyclic triple covers of the projective line.  Theorem \ref{thmain}
provides a crucial base case for forthcoming work of Ozman, Pries and
Weir on such covers \cite{ozmanpriesweir}, and that work was the
initial impetus for the present study.

For a topological space $T$, let $\Pi_0(T)$ denote the set of
irreducible components of $T$.  If $T \subset \stk M$, then $\bar T$
is its closure in $\stk M$.
The symbol $k$ will denote an
arbitrary algebraically closed field of characteristic $p$.

\section{Background on $\stk M$}

\subsection{Moduli spaces}

Let $\stk M$ be the moduli stack (over $\calo_L/p \iso \ff_{p^2}$) of
principally polarized abelian varieties of dimension $n$ with an
action by $\calo_L$ of signature $(1,n-1)$.  Somewhat more precisely,
$\stk M(S)$ consists of isomorphism classes of data
$(X,\iota,\lambda)$, where $X \ra S$ is an abelian variety of relative 
dimension $n$, $\iota: \calo_L \ra \End_S(X)$ is an embedding taking
$1_L$ to $\id_X$ such that ${\rm Lie}(X)$, as a module over
$\calo_L\tensor \calo_S$, has signature $(1,n-1)$; and $\lambda: X \ra
X\dual$ is a principal polarization such that, if $(\dagger)$ is the
induced Rosati involution on $\End(X)$, then for each $a\in \calo_L$
one has $\iota(\bar a) = \iota(a)^{(\dagger)}$.  It is standard that
$\dim \stk M = 1\cdot (n-1) = n-1$.

In fact, $\stk M$ is the moduli stack attached to the Shimura
(pro-)variety constructed from a certain group $\G$, as follows.  

Let $V$ be an $n$-dimensional vector space over $L$, equipped with a
Hermitian pairing of signature $(1,n-1)$.  Let $\G/\rat$ be the group
of unitary similitudes of $V$, and let $\U$ the unitary group of $V$.  Fix
a hyperspecial subgroup $\level_p \subset G(\rat_p)$.  For each
sufficiently small open compact subgroup $\level^p \subset
G(\aff^p_f)$, there is a moduli space $\stk M_{\level^p} = \stk
M_{\level_p\level^p}$ of abelian varieties of dimension $n$ as above
with $\level^p$ structure; see \cite{kottwitz} for more details.  If
$\level^p$ is sufficiently small, then $\stk M_{\level^p}$ is a
smooth, quasiprojective variety; and $\stk M$ may be constructed as
the quotient of any $\stk M_{\level^p}$ by an appropriate finite
group.

\subsection{Newton polygons in $\stk M$}
\label{subsecbacknp}

Newton and Ekedahl-Oort stratifications on $\gu(1,n-1)$ Shimura
varieties are well understood \cite{bultelwedhorn}.    There are
exactly $1 + \floor{n/2}$ (``admissible'') Newton polygons which
occur, and the poset of admissible Newton polygons is actually totally
ordered.  Let $\sigma$ be the supersingular Newton polygon, so that 
$\sigma \preceq \xi$ for any
admissible Newton polygon $\xi$ for $\stk M$.  For a Newton polygon
$\xi$, let $\stk M^\xi$ denote the locally closed locus corresponding
to abelian varieties with Newton polygon $\xi$.  Then $\stk M^\sigma$
is pure of dimension $\floor{\half{n-1}}$.   By purity \cite{dejongoort,nicolevasiuwedhorn}, if $Z_\sigma
\in \Pi_0(\stk M^\sigma)$ and $\sigma \preceq \xi$, then there exists
some $Z_\xi \in \Pi_0(\stk M^\xi)$ such that $Z_\sigma \subseteq
\bar{Z_\xi}$, the closure of $Z_\xi$ in $\stk M$.

The Newton stratification of $\stk M$ is described in
\cite{bultelwedhorn}, as follows.   Each admissible Newton polygon is
determined by its smallest slope.
For each integer $1 \le j \le
\floor{n/2}$, there is a Newton polygon $\xi_{2j}$, with smallest slope
\[
\lambda(2j) = \half 1 - \oneover{2(\floor{n/2}+1-j)};
\]
then $\stk M^{\xi_{2j}}$ has codimension $\floor{n/2}-j$ in $\stk M$.
(Admittedly, in many ways this normalization is more awkward than that
of \cite{bultelwedhorn}, in which $\stk M^{\xi_{2j}}$ is labeled $\stk
M_{2(\floor{n/2}-j)}$; but it will be more convenient for the
deformation theory below.)

Away from the supersingular locus $\stk M^\sigma$, the Newton,
Ekedahl-Oort, and final stratifications coincide; a $p$-divisible
group is determined by its $\bmod p$ truncation \cite[Thm.\
5.3]{bultelwedhorn}.  This is recalled in greater detail in Section
\ref{subsecpdivgp} below.

The Newton polygon and Ekedahl-Oort type of a polarized
$\calo_L$-abelian variety with prime-to-$p$ level structure do not
depend on the level structure, and we set $\stk M_{\level^p}^\xi = \stk
M_{\level^p}\cross_{\stk M} \stk M^\xi$.

\subsection{$p$-divisible groups}
\label{subsecpdivgp}
In contrast to the Siegel case, it is possible to write down a finite,
explicit collection of those principally quasipolarized $p$-divisible
groups with $\calo_L$-action which occur as
$(X,\iota,\lambda)[p^\infty]$ for $(X,\iota,\lambda) \in \stk M(k)$.
Following Wedhorn, we describe such $p$-divisible groups in terms of
their covariant Dieudonn\'e modules, as follows.

For $m \in \nat$, let $M(m)$ be the following Dieudonn\'e module.
\begin{itemize}
\item As a $W(k)$-module, $M(m)$ admits basis $\st{u_1, \cdots, u_m, v_1,
    \cdots, v_m}$.
\item A display \cite{normalg,zinkdisp} for $M(m)$ is
\begin{align*}
F u_1 &= (-1)^mv_m & v_1 &= Vu_m\\
Fv_2 &= u_1 & u_2 &= Vv_1 \\
F v_3 &= u_2 &u_3 &= Vv_2 \\
\vdots &&\vdots \\
F v_m &= u_{m-1} &u_m &= V v_{m-1}
\end{align*}

\item The two eigenspaces for the action of $\calo_L$ on $M(m)$ are
  $\oplus W(k) u_i$ and $\oplus W(k) v_j$.
\item The quasipolarization is given by the symplectic pairing
  $\ang{\cdot,\cdot}$ where
\begin{align*}
\ang{u_i, v_j}& = (-1)^i\delta_{ij}\\
\ang{u_i,u_j} = \ang{v_i,v_j} &= 0
\end{align*}
\end{itemize}

We also define the Dieudonn\'e module $N$:
\begin{itemize}
\item As a $W(k)$-module, $N$ admits basis $\st{u_0,v_0}$.
\item A display for $N$ is
\begin{align*}
Fv_0 &= -u_0 & u_0 &= Vv_0.
\end{align*}
\item The quasipolarization is given by the symplectic pairing
  $\ang{u_0,v_0}=1$.
\item The two eigenspaces for the action of $\calo_L$ are  $W(k)u_0$
  and $W(k)v_0$.
\end{itemize}

With this notation in place, one has the following result of B\"ultel
and Wedhorn:

\begin{theorem}\cite{bultelwedhorn}
\label{thbultelwedhorn}
\begin{alphabetize}
\vspace{\baselineskip}
\item Suppose $1 \le j \le \floor{n/2}$.  There exists an integer
  $r(j)$ such that, if $(X,\iota,\lambda) \in
  \stk M^{\xi_{2j}}(k)$, then
\[
\dieu_*((X,\iota,\lambda)[p^\infty]) \iso  M(2(\floor{n/2}+1-j) \oplus
N^{r(j)}.
\]
\item There exists an open dense subspace $\stk M^{\sigma \circ}
  \subset \stk M^\sigma$ such that, if $(X,\iota,\lambda) \in \stk
  M^{\sigma\circ}(k)$, then
\[
 \dieu_*((X,\iota,\lambda)[p^\infty])
\iso
\begin{cases}
M(n) & n\text{ odd} \\
M(n-1) \oplus  N & n\text{ even}
\end{cases}
.
\]
\end{alphabetize}
\end{theorem}

\subsection{Hecke operators}

An inclusion $\level_1^p \inject \level_2^p$ of open compact subgroups
of $\G(\aff^p_f)$ induces a cover of Shimura varieties $\stk
M_{\level_1^p} \ra \stk M_{\level_2^p}$. More generally, an element $g
\in \G(\aff^p_f)$ induces, for each open compact $\level^p$, a natural
morphism  $\stk M_{\level^p} \ra \stk M_{g\inv \level^pg}$.  

Let $z \in \stk M_{\level_0^p}(k)$.  Its prime-to-$p$ (unitary) Hecke
orbit, $\calh^p(z)$, is defined as follows.  Consider the pro-variety
$\myht{\stk M}_{\level_0^p} = \invlim{\level^p\subset \level_0^p} \stk
M_{\level^p}$.  Choose a lift $\hat z$ of $z$ to $\myht {\stk
M}_{\level_0^p}$.  Then $\calh^p(z)$ is the projection to $\stk
M_{\level_0^p}$ of $\U(\aff^p_f) \hat z$.  (One can also
construct the ``similitude'' Hecke orbit of $\hat z$, by replacing the
orbit $\U(\aff^p_f)\hat z$ with $\G(\aff^p_f)\hat z$.  However, the
unitary Hecke orbit is both the output of \cite[Thm.\ 4.6]{yumasspel}
and the input to \cite[Thm.\ 1.4]{kasprowitz}, and thus better suited
to the task at hand.)

\section{Closures of Newton strata}

Let $\xi$ be an admissible Newton polygon for $\stk M$ such that $\xi
\not = \sigma$.

\begin{lemma}
\label{lemxismooth}
The locus $\stk M^\xi$ is smooth.
\end{lemma}

\begin{proof}
The isomorphism class of $(X[p^\infty],\iota[p^\infty],
\lambda[p^\infty])$ for $(X,\iota,\lambda) \in \stk M^\xi(k)$ is
independent of the choice of point (Theorem \ref{thbultelwedhorn}).
By the Serre-Tate theorem, the formal neighborhoods of all points of
$\stk M^\xi$ are thus isomorphic.  Since $\stk M^\xi$ is by
definition equipped with the reduced subscheme structure, it must be
smooth.
\end{proof}

\begin{lemma}
\label{lemxidegenerates}
If $Z_\xi \in \Pi_0(\stk M^\xi)$, then there exists $Z_\sigma \in
\Pi_0(\stk M^\sigma)$ such that $Z_\sigma \subset \bar Z_\xi$.
\end{lemma}

\begin{proof}
We prove the following apparently stronger result.  Suppose $\nu$ and
$\xi$ are admissible Newton polygons with $\nu \prec \xi$, and $Z_\xi
\in \Pi_0(\stk M^\xi)$.  We show that there exists $Z_\nu \in \Pi_0(\stk
N^\sigma)$ such that $Z_\nu \subset \bar Z_\xi$.  It suffices to prove
this statement under the assumption that $\nu$ is the immediate
predecessor of $\xi$, so that $\dim \stk M^\nu = \dim \stk M^\xi-1$.
The statement is trivially true if $\xi = \xi_{2\floor{n/2}}$ is the
locus with positive $p$-rank; henceforth, we assume that $\xi$ is
strictly smaller than $\xi_{2\floor{n/2}}$.

It is is slightly more convenient to work with fine moduli schemes.
Let $\level^p \subset \G(\aff^p_f)$ be an open compact subgroup which
is small enough that $\stk M_{\level^p}$ is a smooth, quasiprojective
variety.  Let $W_\xi \in \Pi_0(Z_\xi \cross_{\stk M} \stk
M_{\level^p})$ be an irreducible component of $\stk M_{\level^p}^\xi$
lying over $Z_\xi$.  It suffices to show that the closure of $W_\xi$
in $\stk M_{\level^p}$ contains an irreducible component of $\stk
M_{\level^p}^\nu$.

Let $\tilstk M_{\level^p}$ be a toroidal compactification of
$\stk M_{\level^p}$ (e.g., \cite[6.4.1.1]{lanbook}).  It is a smooth,
projective variety.  Let $\til W_\xi$ be the closure of $W_\xi$ in
$\tilstk M_{\level^p}$, and let $\del W_\xi = \til W_\xi \setcomp W_\xi$.
Newton strata (other than the supersingular stratum) coincide with Ekedahl-Oort
strata, and the latter are known to be affine (e.g.,
\cite{nicolevasiuwedhorn}).   Because $W_\xi$ is positive dimensional,
$\del W_\xi$ is nonempty. The first slope of $\xi$ is positive, while
the boundary of $\tilstk M_{\level^p}$ parametrizes semiabelian
varieties with nontrivial toric part.  Consequently, $\del W_\xi \cap
(\tilstk M_{\level^p}\setcomp \stk M_{\level^p})$ is empty, and $\del W_\xi
\subset \stk M_{\level^p}$.  Again by purity
(\cite{nicolevasiuwedhorn}), $\dim \del W_\xi = \dim W_\xi -1$.  By
semicontinuity of Newton polygons, there is an a priori
containment $\del W_\xi \subseteq \cup_{\tau \prec \xi} \stk
M_{\level^p}^{\tau}$.  The result now follows from dimension counts:  $\stk
M_{\level^p}^\nu$ is pure of dimension $\dim W_\xi -1$, while if $\tau
\prec \nu$ then $\dim \stk M_{\level^p}^\tau <\dim W_\nu = \dim \del
W_\xi$.
\end{proof}

Conversely, 
\begin{lemma}
\label{lemsigmainboundary}
If  $Z_\sigma \in \Pi_0(\stk M^\sigma)$, then there is a unique $Z_\xi
\in \Pi_0(\stk M^\xi)$ such that $Z_\sigma \subset \bar Z_\xi$.
\end{lemma}

\begin{proof}
  The existence of such a $Z_\xi$ follows from purity and
  dimension-counting.  If there
  were two such components, then they would intersect along $Z_\sigma
  \cap \stk M^{\sigma\circ}$, which would contradict the smoothness
  shown in Lemma \ref{lemlocallysmooth}.
\end{proof}

\section{Local calculations}
\label{subseclocal}

\begin{lemma}
\label{lemlocallysmooth}
Let $\xi$ be an admissible Newton polygon which is not supersingular,
and suppose $z \in \stk M^{\sigma\circ}(k)$.  Then $\barstk M^\xi$
is smooth at $z$.
\end{lemma}

\begin{proof} 
This follows directly from the explicit calculation (Lemmas
\ref{lemlocalnpstrataodd} and \ref{lemlocalnpstrataeven}) of the
Newton stratification on the formal
neighborhood $\stk M^{/z}$ of $z$.
\end{proof}

The necessary calculations are somewhat sensitive to the parity of
$n$.  We first work out the details when $n$ is odd, and then indicate
the changes necessary to accommodate even $n$.

\subsection{The case of $n$ odd}
\label{subsubnodd}

Throughout this section, assume that $n$ is odd.

\subsubsection{Explicit deformations}
\label{subseclocalcoords}

Suppose $z = (X,\iota,\lambda) \in \stk M^{\sigma\circ}(k)$.   Our
goal is to understand the Newton stratification on the formal
neighborhood $\stk M^{/z} = \spf \til R$ of $z$ in $\stk M$.  This will be
accomplished using (covariant) Dieudonn\'e theory.  
Suppose $z = (X,\iota,\lambda) \in \stk M^{\sigma\circ}(k)$.  Then
the Dieudonn\'e module $\dieu_*(X[p^\infty])$ is isomorphic to $M :=
M(n)$ (Theorem \ref{thbultelwedhorn}).

Deformations of $X[p^\infty]$ are parametrized by $\hom(VM/pM,
M/VM)$; those which preserve the $\calo_L$-structure are classified by
$\hom_{\calo_L\tensor k}(VM/pM, M/VM)$ (e.g., \cite{achterrmjm}).  The
display we have chosen gives coordinates on  
 $VM/pM$ and $M/VM$:
\begin{align*}
VM/pM &= k \st{v_1, u_2, u_3, \cdots, u_n} \\
M/VM &= k \st{u_1, v_2, v_3, \cdots, v_n}
\end{align*}
Consequently,
\begin{align*}
\hom_{\calo_L\tensor k}(VM/pM,M/VM) &= k\st{ v_1^*v_2, v_1^*v_3,
  \cdots, v_1^*v_n, u_2^*u_1, u_3^*u_1, \cdots, u_n^*u_1} \\
& \subset \hom_k(VM/pM,M/VM)  = (VM/pM)^* \tensor (M/VM),
\end{align*}
and the universal equicharacteristic deformation ring of
$(X[p^\infty],\iota[p^\infty])$ is
\[
\til R'  = k\powser{t(v_1v_2), t(v_1v_3), \cdots, t(v_1v_n), t(u_2u_1),
  \cdots, t(u_nu_1)}.
\]
For $t(xy) \in \til R'$, let $\ubar {t}(xy)$ be its Teichmuller
lift to $W(\til R')$.  Then $\til X[p^\infty]$ is displayed over $\til R'$ by
\begin{align*}
\til F u_1 &= -v_n & v_1 &= \til V(u_n + \ubar t(u_nu_1)u_1)\\
\til Fv_2 &= u_1 & u_2 &= \til V(v_1 +\sum_{2 \le j \le n} \ubar t(v_1v_j)v_j)\\
\til F v_3 &= u_2 + \ubar t(u_2u_1)u_1& u_3 &= \til Vv_2 \\
\vdots &&\vdots \\
\til F v_n &= u_{n-1}+ \ubar t(u_{n-1}u_1)u_1 &u_n &= \til V v_{n-1}
\end{align*}
The pairing $\ang{\cdot,\cdot}$ extends to $\til M = M \tensor_{W(k)}
W(\til R')$ by linearity.
We would like to identify the largest quotient $\til R$ of $\til R'$
to which $\ang{\cdot,\cdot}$ extends as a pairing of Dieudonn\'e
modules; for then $\stk M^{/z} \iso \spf \til R$.

The quasipolarization extends to a ring $R$ if and only if, for each $x,y \in
\til M_R := \til M \tensor_{W(\til R')}W(R)$, one has
\[
\ang{\til Fx,y} = \ang{x, \til V}^\sigma.
\]

Suppose $3 \le j \le n$, and let $(x,y) = (v_j, v_1 + \sum_{2 \le k
  \le n} \ubar t(v_1v_k)v_k)$.  Then
\begin{align*}
\ang{\til F x, y} &= \ang{ u_{j-1} + \ubar t(u_{j-1}u_1)u_1, v_1 + \sum_{2
    \le k \le n} \ubar t(v_1v_k)v_k} \\
&= \ubar t(v_1v_{j-1})\ang{u_{j-1}, v_{j-1}} + \ubar t(u_{j-1}u_1)\ang{u_1, v_1}
\\
&= (-1)^{j-1}\ubar t(v_1v_{j-1}) - \ubar t(u_{j-1}u_1),
\intertext{while}
\ang{x, \til Vy} &= \ang{v_j,u_2} \\
&= 0.
\end{align*}
Consequently, if $\til M_R$ is  quasipolarized by
$\ang{\cdot,\cdot}$, then for each $2 \le k \le n-1$, the image of
$(-1)^k t(v_1v_k) - t(u_ku_1)$ in $R$ is zero.

Similarly, by considering $(x,y) = (v_1, v_2 + \sum \ubar t(v_2v_j)v_j)$, we
see that the image of $t(v_1v_n)$ in such an $R$ must be zero.

The quotient $\til R$ of $\til R'$ by these relations is a smooth, local ring of
dimension $n-1$, and thus we identify $\til R$ with
\[
\til R = k\powser{s_2, \cdots, s_n},
\]
where $s_j$ is the image of $t(u_ju_1)$ in $\til R$.   We record these
calculations as follows.

\begin{lemma}
The formal neighborhood $\stk M^{/z}$ of $z$ is isomorphic to 
$\til R = \spf k\powser{s_2, \cdots, s_n}$.  Over $\til R$, the
Dieudonn\'e module $\til M = \til M_{\til R}$ of the universal deformation of $(X[p^\infty], \iota[p^\infty],
\lambda[p^\infty])$ is displayed by
\begin{align*}
\til F u_1 &= -v_n & v_1 &= \til V(u_n  - \ubar s_nu_1)\\
\til Fv_2 &= u_1 & u_2 &= \til V(v_1 +\sum_{2 \le j \le n} (-1)^j \ubar s_jv_j)\\
\til F v_3 &= u_2 + \ubar s_2u_1& u_3 &= \til Vv_2 \\
\vdots &&\vdots \\
\til F v_n &= u_{n-1}+ \ubar s_{n-1}u_1 &u_n &= \til V v_{n-1}
\end{align*}
\end{lemma}

\subsubsection{Newton strata in local coordinates}
\label{subseclocalnp}

In this choice of coordinates, it is easy to calculate the Newton
stratification on $\stk M^{/z}$.  For $1 \le j \le \floor{n/2}$, let
\[
\stk M^{/z}_{<j} := \stk M^{/z} \cap (\stk M^\sigma \cup \bigcup_{1 \le i <
  j} \stk M^{\xi_{2i}})
\]
be the locus in $\stk M^{/z}$ parametrizing those deformations whose
first slope is strictly larger than $\lambda(2j)$.

\begin{lemma}
\label{lemlocalnpstrataodd}
Suppose $1 \le j \le \floor{n/2}$.  Then
\[
\stk M^{/z}_{<j} = \spf \frac{k\powser{s_2, \cdots, s_n}}{(s_{2j}, s_{2(j+1)}, \cdots,
  s_{2\floor{n/2}})} \subset \stk M^{/z} = \spf k\powser{s_2, \cdots, s_n}.
\]
\end{lemma}

\begin{figure}
\label{figgamma}
\begin{center}
\begin{tikzpicture}[yscale=0.5,>=stealth]
\foreach \x in {1,...,7}
{
  \filldraw[black] (\x,0) circle (0.1cm);
  \node [above] at (\x,2) {$u_{\x}$};
  \filldraw[black] (\x,2) circle (0.1cm);
  \node [below] at (\x,0) {$v_{\x}$};
}
\begin{scope}[very thick]
\foreach \x in {2,3,...,7}
 \draw[->>] (\x,2) -- ++(-1,-2);
\draw[->>] (1,0) .. controls (2,-2) and (9,-2) .. (7,2);
\end{scope}
\begin{scope}[thin]
\foreach \x in {2,3,...,7}
 \draw[->>] (\x,0) -- ++(-1,2);
\draw [->>](1,2) ..  controls (-1,-2) and (5,-2) .. (7,0);
\end{scope}
\end{tikzpicture}
\end{center}
\caption{The graph $\Gamma$ for $n=7$.}
\end{figure}
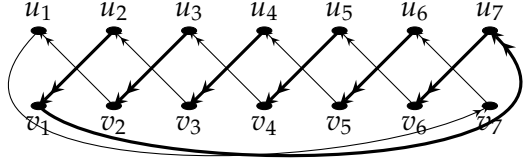

Before proceeding with the proof, we construct a graph to encode part
of the structure of (a deformation of) $M$. Initially, construct a
graph $\Gamma$ as follows (see Figure \ref{figgamma}).  With a slight
abuse of notation, let the vertex set be $\st{u_1, \cdots, u_n, v_1,
  \cdots, v_n}$.  For $2 \le i \le n$, draw a (light) gray arrow from
$v_i$ to $u_{i-1}$, to encode the fact that $Fv_i = u_{i-1}$.
Similarly, draw a gray arrow from $u_1$ to $v_n$.

Also, for each $2 \le i \le n$, draw a black arrow from $u_i$ to
$v_{i-1}$, to encode the fact that $Fu_i = pv_{i-1}$.  Similarly, draw
a black arrow from $v_1$ to $u_n$.

Note that $\Gamma$ is a (colored) cycle.  In fact, starting from
vertex $u_1$, one successively visits 
\[
\st{u_1,v_n,u_{n-1},
  v_{n-2},u_{n-3}, \cdots, v_1, u_n, v_{n-1}, u_{n-2}, \cdots, v_2,
  u_1}.
\]

Now let $S$ be an integral domain equipped with a surjection $\phi:k\powser{s_2,
  \cdots, s_n} \ra S$, and let $K$ be the field of fractions of $S$.
Construct a graph $\Gamma_S$ by (possibly) augmenting the edge set of
$\Gamma$, as follows.

For each $2 \le i \le n-1$, if $\phi(s_i)\not = 0$, then add a gray edge
from $v_j$ to $u_1$.   (For the sake of completeness, if $\phi(s_n)\not =0$,
  then add a black edge from $v_1$ to $u_1$.  For each $2 \le i \le
  n-1$, if $\phi(s_i)\not = 0$, then add a black edge from $u_2$ to
  $v_i$.  These additional black edges will not affect the final calculation.)

Let $C$ be a cycle or path in $\Gamma_S$.  The length of $C$ is the
number of edges in $C$, while the weight of $C$ is the number of black
edges in $C$.  Define the slope of $C$ to be
\[
\lambda(C) = \frac{{\rm weight}(C)}{{\rm length}(C)}.
\]

Note that for the trivial deformation, corresponding to $\Gamma$
itself, we have $\lambda(\Gamma) = \frac{n}{2n} = \half 1$.

\begin{lemma}
\label{lemslopeest}
If $C \subset \Gamma_S$ is a cycle through $u_1$, then the smallest
slope of the Newton polygon is at most $\lambda(C)$.
\end{lemma}

\begin{proof} It is harmless, and convenient, to replace $K$
  by its perfection.  Suppose there is a cycle $C$ of
  length $b$ and weight $a$; let $\til N_K = W(K)\st{u_2, \cdots, u_n, v_1,
    \cdots, v_n}$.   Then $F^b u_1  \in p^aW(K)\st{u_1} + \til N_K$ but
  $F^bu_1 \not \in \til N_K$, and $\til M_K/\til N_K$ is an $F$-$\sigma^a$-crystal
  of slope at most $a/b$.  Therefore, the smallest slope of $\til M_K$ is
  at most $a/b$.
\end{proof}

\begin{remark}
Let $B(K) =
\Frac W(K)$; then the $B(K)[F]$-span
of $u_1$ in $\til M_K \tensor B(K)$ is all of $\til M_K\tensor B(K)$.
Therefore, one can in fact 
show that the smallest slope of $\til M_K$ is
\[
\min_{C \subset \Gamma\text{ a cycle through }u_1} \lambda(C).
\]
\end{remark}

\begin{lemma}
\label{lemgammascycle}
If $\phi(s_{2i})\not = 0$, then there is a cycle in $\Gamma_S$ of
length $n+1-2j$ and weight $\half{n-1}-j$.
\end{lemma}

\begin{proof}
In $\Gamma$, the unique path $P$ from $u_1$ to $v_{2j+1}$ has length
$n+1-(2j+1)=n-2j$ and weight $\half{n-(2j+1)} = \half{n-1}-j$.  If
$\phi(s_{2j})\not = 0$, then in $\Gamma_S$ there is a cycle, obtained
by concatenating $u_1$ to $P$, of length ${\rm length}(P)+1$ and
weight ${\rm weight}(P)$.
\end{proof}

\begin{lemma}
\label{lemapproxequations}
If the smallest slope of $\til M_K$ is greater than $\lambda(2j)$, then
\[
\phi(s_{2j}) = \phi(s_{2(j+1)}) = \cdots = \phi(2\floor{n/2}) =0.
\]
\end{lemma}

\begin{proof}
The contrapositive follows immediately from Lemmas
\ref{lemgammascycle} and \ref{lemslopeest}; if there is some $i
\ge j$ with $\phi(s_{2i})\not = 0$, then the smallest slope of $\til M_K$
is at most $\lambda(2i)$.
\end{proof}

\begin{proof}[Proof of Lemma \ref{lemlocalnpstrataodd}]
By Lemma \ref{lemapproxequations}, the sought-for neighborhood $\stk
N^{/z}_{<j}$ is the
formal spectrum of a quotient of $R_{<j} := k\powser{s_1, \cdots, s_n}/(s_{2j},
s_{2(j+1)},  \cdots, s_{2\floor{n/2}})$.  We thus have $\stk
N^{/z}_{<j} \inject \spf R_{<j} \inject \stk M^{/z}$.  Since both
$\stk M^{/z}_{<j}$ and $\spf R_{<j}$ have codimension
$\floor{n/2}-j+1$, the result follows.
\end{proof}

\subsection{The case of $n$ even}

We now indicate the changes which must be made in order to perform the
calculations of Section \ref{subsubnodd} in the case where $n$ is even.

Suppose $z= (X,\iota,\lambda) \in \stk M^{\sigma\circ}(k)$.  The
quasipolarized Dieudonn\'e module $M$ of $X[p^\infty]$, as a $p$-divisible
group with $\calo_L$-action, is $M(n-1)\oplus N$ (Theorem
\ref{thbultelwedhorn}).  A calculation exactly like that in
Section \ref{subseclocalcoords} shows $\stk M^{/z} \iso \spf \til R =
\spf k\powser{s_0, s_2, \cdots, s_{n-1}}$; the corresponding
deformation $\til M$ of $M$ is displayed by
\begin{align*}
\til F v_0 &= -u_0-\ubar s_0u_1 & u_0 &= \til V v_0 \\
\til F u_1 &= -v_{n-1} & v_1 &= \til V(u_{n-1}  - \ubar s_{n-1}u_1)\\
\til Fv_2 &= u_1 & u_2 &= \til V(v_1 +\sum_{2 \le j \le n-1} (-1)^j \ubar s_jv_j)\\
\til F v_3 &= u_2 + \ubar s_2u_1& u_3 &= \til Vv_2 \\
\vdots &&\vdots \\
\til F v_{n-1} &= u_{n-2}+  \ubar s_{n-2}u_1 &u_{n-1} &= \til V v_{n-2}
\end{align*}
Construction and analysis of graphs $\Gamma$ and $\Gamma_S$, for
quotients $S$ of $\til R$, shows that Lemma \ref{lemlocalnpstrataodd} holds
for even $n$, too:

\begin{lemma}
\label{lemlocalnpstrataeven}
Suppose $1 \le j \le n/2$.  Then
\[
\stk M^{/z}\cap (\cup_{1 \le i \le j} \stk M^{\xi_{2i}}) = \spf
\frac{k\powser{s_0, s_2, s_3, \cdots, s_n}}{(s_{2j},s_{2(j+1)},
  \cdots, s_{n})}.
\]
\end{lemma}

\section{Hecke orbits for the supersingular locus}

\begin{lemma}
\label{lemheckesupersingular}
Let $\level^p\subset G(\aff^p_f)$ be a compact open subgroup.
The $\U(\aff^p_f)$-Hecke operators act transitively on $\Pi_0(\stk
M_{\level^p}^\sigma)$.
\end{lemma}

\begin{proof}
  Let $(X,\iota,\lambda) = \bar\eta$ be a geometric generic point of
  $\stk M_{\level^p}^\sigma$.  The central leaf $\calc([\bar\eta])$, which in a
  general PEL Shimura variety context parametrizes those
  $(Y,\jmath,\mu)$ with $(Y[p^\infty], \jmath[p^\infty],\mu[p^\infty])
  \iso (X[p^\infty],\iota[p^\infty],\lambda[p^\infty])$, in this case
  coincides with (the union of a choice of geometric point over each
  generic point of) $\stk M^{\sigma\circ}$.  There is an a priori
  inclusion $\calh^p(\bar\eta) \subseteq \calc([\bar\eta])$.  Since
  $\bar\eta$ is basic and $\G$, the reductive group defining $\stk M$,
  has simply connected derived group, the prime-to-$p$ Hecke orbit of
  $\bar\eta$ coincides with the central leaf $\calc([X[p^\infty],
  \iota[p^\infty], \lambda[p^\infty]])$ \cite[Thm.\ 4.6(1) and Rem.\
  4.7(3)]{yumasspel}.
\end{proof}

\section{Irreducibility of Newton strata}

\begin{proof}[Proof of Theorem \ref{thmain}]
Chai and Oort identify nine steps in their proof of
\cite[Thm.\ 3.1]{chaioort11}, which is the analogue for $\stk A_g$ of
Theorem \ref{thmain}.  We proceed here in a similar fashion.  Fix an
open compact subgroup $\level^p \subset \G(\aff^p_f)$; it suffices to
prove that $\stk M_{\level^p}^\xi$is irreducible.

\begin{description}
\item[Steps 1-6] By Lemma \ref{lemsigmainboundary}, there is a
  well-defined map of sets
\begin{diagram}
\Pi_0(\stk M^\sigma) & \rto  & \Pi_0(\stk M^\xi).
\end{diagram}
It is surjective, by Lemma \ref{lemxidegenerates}.  From it, we deduce
the existence of a surjective
\begin{diagram}
\Pi_0(\stk M_{\level^p}^\sigma) & \rto & \Pi_0(\stk M_{\level^p}),
\end{diagram}
visibily $\U(\aff^p_f)$-equivariant.

\item[Steps 7-8] By Lemma \ref{lemheckesupersingular}, the action of
  $\U(\aff^p_f)$ on  $\Pi_0(\stk M_{\level^p}^\sigma)$ is transitive.

\item[Step 9] Taken together, this shows that $\U(\aff^p_f)$ acts
  transitively on $\Pi_0(\stk M^\zeta_{\level^p})$.  By \cite[Thm.\
  1.4]{kasprowitz}, which is the PEL analogue of \cite{chailadic},
  $\stk M_{\level^p}^\xi$ is connected.  Since $\stk M_{\level^p}^\xi$ is also smooth
  (Lemma \ref{lemxismooth}), it is irreducible.
\end{description}
\end{proof}

\bibliographystyle{hamsplain}
\bibliography{jda}

\def\cprime{$'$}
\providecommand{\bysame}{\leavevmode\hbox to3em{\hrulefill}\thinspace}
\providecommand{\href}[2]{#2}
\begin{thebibliography}{10}

\bibitem{achterrmjm}
Jeffrey~D. Achter, \emph{Hilbert-{S}iegel moduli spaces in positive
  characteristic}, Rocky Mountain J. Math. \textbf{33} (2003), no.~1, 1--25.

\bibitem{bultelwedhorn}
Oliver B{\"u}ltel and Torsten Wedhorn, \emph{Congruence relations for {S}himura
  varieties associated to some unitary groups}, J. Inst. Math. Jussieu
  \textbf{5} (2006), no.~2, 229--261.

\bibitem{chailadic}
Ching-Li Chai, \emph{Monodromy of {H}ecke-invariant subvarieties}, Pure Appl.
  Math. Q. \textbf{1} (2005), no.~2, 291--303.

\bibitem{chaioort11}
Ching-Li Chai and Frans Oort, \emph{Monodromy and irreducibility of leaves},
  Ann. of Math. (2) \textbf{173} (2011), no.~3, 1359--1396,
  {10.4007/annals.2011.173.3.3}.

\bibitem{dejongoort}
A.~J. de~Jong and F.~Oort, \emph{Purity of the stratification by {N}ewton
  polygons}, J. Amer. Math. Soc. \textbf{13} (2000), no.~1, 209--241.

\bibitem{kasprowitz}
Ralf Kasprowitz, \emph{Monodromy of subvarieties of {P}{E}{L}-{S}himura
  varieties}, arXiv preprint arXiv:1209.5891 (2012).

\bibitem{kottwitz}
Robert~E. Kottwitz, \emph{Points on some {S}himura varieties over finite
  fields}, J. Amer. Math. Soc. \textbf{5} (1992), no.~2, 373--444.

\bibitem{lanbook}
Kai-Wen Lan, \emph{Arithmetic compactifications of {P}{E}{L}-type {S}himura
  varieties}, London Math. Soc. Monographs, no.~36, Princeton University Press,
  Princeton, NJ, 2013.

\bibitem{nicolevasiuwedhorn}
Marc-Hubert Nicole, Adrian Vasiu, and Torsten Wedhorn, \emph{Purity of level
  {$m$} stratifications}, Ann. Sci. \'Ec. Norm. Sup\'er. (4) \textbf{43}
  (2010), no.~6, 925--955.

\bibitem{normalg}
Peter Norman, \emph{An algorithm for computing local moduli of abelian
  varieties}, Ann. Math. (2) \textbf{101} (1975), 499--509.

\bibitem{ozmanpriesweir}
Ekin Ozman, Rachel Pries, and Colin Weir, \emph{The $p$-rank of cyclic covers
  of the projective line}, 2014, in preparation.

\bibitem{yumasspel}
Chia-Fu Yu, \emph{Simple mass formulas on {S}himura varieties of {PEL}-type},
  Forum Math. \textbf{22} (2010), no.~3, 565--582, {10.1515/FORUM.2010.030}.

\bibitem{zinkdisp}
T.~Zink, \emph{The display of a formal $p$-divisible group}, Ast\'erisque
  (2002), no.~278.

\end{thebibliography}

\end{document}